\documentclass{article}
\usepackage{geometry}
\usepackage{titling}
\AtEndDocument{\bigskip{\footnotesize%
\textsc{$^*$School of Computing, DIT University, Dehradun, Uttarakhand–248009, India} \par  
  \textit{E-mail address:} \texttt{suryagiri456@gmail.com} \par

  \addvspace{\medskipamount}
}}
\usepackage{doc}

\usepackage{url}

\usepackage{graphicx}
\usepackage{epstopdf}
\usepackage{amsthm}
\usepackage{amssymb}
\usepackage{amsmath}
\usepackage{enumitem}

\usepackage{array}
\usepackage{enumitem}
\usepackage[utf8]{inputenc}
\usepackage[english]{babel}
\newtheorem{theorem}{Theorem}
\newtheorem{corollary}{Corollary}[theorem]
\newtheorem{lemma}[theorem]{Lemma}
\newtheorem*{definition}{Definition}
\newtheorem{defn}{Definition}%

\newtheorem*{assumption}{Assumption}

\DeclareMathOperator{\RE}{Re}

\usepackage{varwidth}
\usepackage{lineno, hyperref}
\usepackage{graphicx}
\usepackage{epstopdf}
\usepackage{amsmath}
\usepackage{amsthm}
\usepackage{enumitem}
\usepackage{amsthm}
\usepackage{float}
\usepackage{subcaption}
\usepackage{caption}
\usepackage{authblk}
\usepackage{abstract}

\begin{document}
\title{Second-Order Toeplitz Determinant for Quasi-Convex Mappings}
\author{Surya Giri$^{*}$  }


\date{}


	

\maketitle	

\begin{abstract}
    \noindent  This paper presents sharp estimates for the second-order Toeplitz determinant whose entries are the coefficients of convex functions defined on the unit disk in $\mathbb{C}$. These estimates are further extended to a subclass of holomorphic mappings defined on the  unit ball in a complex Banach space and on the unit polydisk in $\mathbb{C}^n$, which, as special cases, yield bounds for the classes of quasi-convex mappings of type $B$.
\end{abstract}
\vspace{0.5cm}
\section{Introduction}
   Suppose that $\mathcal{A}$ represents the collection of all analytic functions defined on the open unit disk $\mathbb{U}$ having the series expansion
\begin{equation}\label{one}
    f(z)= z + \sum_{n=2}^\infty a_n z^n, \quad z \in \mathbb{U}.
\end{equation}
    Let $\mathcal{S} \subset \mathcal{A}$ comprise all univalent functions. Various subclasses of $\mathcal{S}$ were introduced and studied in view of their geometric significance. One well studied subclass of $\mathcal{S}$ is the class of convex functions, denoted by $\mathcal{C}$. Let $f \in \mathcal{S}$. Then $f$ is said to be convex if $f(\mathbb{U})$ is a convex domain. Analytically, $f\in \mathcal{C}$ if and only if
    $$ \RE \Big(1 + \frac{z f''(z)}{f'(z)} \Big) >0 , \quad z\in \mathbb{U}. $$
    A function $f\in \mathcal{A}$ is said to be subordinate to another function $F\in \mathcal{A}$ if there exists a Schwarz function $\omega$ such that $f(z)=F(\omega(z))$. In notation, we write $f\prec F$. Let $\mathcal{B}_0$ denote the class of all Schwarz functions. Using the concept of subordination, Ma and Minda~\cite{MaMi} introduced the class
    $$ \mathcal{C}(\Psi) = \bigg\{ f\in \mathcal{A}: 1 + \frac{z f''(z)}{f'(z)} \prec \Psi(z) \bigg\}, $$
    where $\Psi : \mathbb{U} \rightarrow \mathbb{C}$ is an analytic univalent functions, such that $\RE \Psi(z)>0$, $\Psi(0)=1$, $\Psi'(0)>0$, $\Psi(\mathbb{U})$ is starlike with respect to $1$ as well as symmetric about the real axis. For specific choices of $\Psi$, the class $\mathcal{C}(\Psi)$ reduces to well-known subclasses of $\mathcal{C}$. For instance, choosing
    $$ \Psi(z)= \frac{1+z}{1-z}, \;\; \Psi(z)= \frac{1+(1-2\alpha)z}{1-z}\;\; \text{and}\;\; \Psi(z)=\Big(\frac{1+z}{1-z}\Big)^\beta$$
    yields the classes of convex functions, convex functions of order $\alpha$ (denoted by $\mathcal{C}(\alpha)$), and strongly convex functions of order $\beta$ (denoted by $\mathcal{CC}(\beta)$), respectively, where $\alpha, \beta \in [0,1]$.

    Ye and Lim~\cite{YeLim} showed that any $n \times n$ matrix can be expressed as a product of Toeplitz matrices.
    Toeplitz matrices arise in  algebra~\cite{Ber}, image processing ~\cite{Coc}, and other areas. For more applications, see~\cite{YeLim}. Ali et al.~\cite{AliThoVas} considered the Toeplitz determinants for $f(z)=z+\sum_{n=2}^\infty a_n z^n \in \mathcal{A}$ given by
\begin{equation*}
     T_{m,n}(f)= \begin{vmatrix}
	a_n & a_{n+1} & \cdots & a_{n+m-1} \\
	a_{n+1} & a_n & \cdots & a_{n+m-2}\\
	\vdots & \vdots & \ddots & \vdots\\
    a_{n+m-1} & a_{n+m-2} & \cdots & a_n\\
	\end{vmatrix}.
\end{equation*}
   Particularly, they derived the sharp estimates of
\begin{equation*}\label{T31}
    \vert T_{3,1}(f)\vert =
      \vert a_2^{2} {a}_3- 2 a_2^2-  a_3^2+1\vert \;\; \text{and} \;\; \vert T_{2,n}(f)\vert= \vert a_n^2 -a_{n+1}^2 \vert
\end{equation*}
  when $f \in \mathcal{C}$ and for certain other subclasses of $\mathcal{S}.$ Ahuja et al.~\cite{AhuKhaRav} investigated the same problem for the class $\mathcal{C}(\Psi)$ and established bounds for $\vert T_{2,2}(f)\vert$ and $\vert T_{3,1}(f)\vert$. 
  Many authors contributed to this area in recent years, see~\cite{GirKum,GirKum3,VasLecTho} and references cited therein. For functions $f\in \mathcal{C}(\Psi)$, the problem of estimating
\begin{equation}\label{Tplz}
    \vert T_{2,3}(f)\vert  = \vert a_3^2 - a_4^2 \vert
\end{equation}
   was open.  In section \ref{1st}, we provide the sharp estimate for $ \vert T_{2,3}(f)\vert$ when $f\in\mathcal{C}(\Psi)$.

   Furthermore, these bounds are extended from one dimension to higher dimensions for a subclass of holomorphic mappings defined on the unit ball of a complex Banach space and on the unit polydisk in $\mathbb{C}^n$. Let $\mathbb{B}$ be the unit ball in a complex Banach space $X$ equipped with a norm $\|\cdot\|$ and $\mathbb{U}^n$ denote the unit poly disk in $\mathbb{C}^n$. Also, let $\partial \mathbb{U}^n$ and $\partial_0 \mathbb{U}^n$ represent the boundary and distinguished boundary of $\mathbb{U}^n$, respectively.

   Let $\mathcal{H}(\mathbb{B})$ be the set of all holomorphic mappings from $\mathbb{B}$ into $X$.
   If $f\in \mathcal{H}(\mathbb{B})$, then for each integer $k \geq 1$, there exists a bounded symmetric linear mapping
    $$D^k f(z) : \prod_{j=1}^k X \rightarrow X, $$
    called the $k^{th}$ order Fr\'{e}chet derivative of $f$ at $z$. Moreover,  for every $w$ in some neighborhood of $z$, the mapping $f$ can be expressed as
    $$ f(w) = \sum_{k=0}^\infty \frac{1}{k!} D^k f(z) ((w -z)^k ),$$
   where $ D^0 f(z) ((w -z)^0 )= f(z)$
    and for $k \geq 1$,
    $$ D^k f(z)( (w -z)^k) = D^k f(z) \underbrace{( w-z, w-z, \cdots, w-z) }_\text{ k -times}.$$
    On a bounded circular domain $\Omega \subset \mathbb{C}^n$, the first and the $m^{th}$ Fr\'{e}chet derivative of a holomorphic mapping $f : \Omega \rightarrow X$  are written by
    $ D f(z)$ and $D^m f(z) (a^{m-1},\cdot)$, respectively. The matrix representations are
\begin{align*}
    D f(z) &= \bigg(\frac{\partial f_j}{\partial z_k} \bigg)_{1 \leq j, k \leq n}, \\
    D^m f(z)(a^{m-1}, \cdot) &= \bigg( \sum_{p_1,p_2, \cdots, p_{m-1}=1}^n  \frac{ \partial^m f_j (z)}{\partial z_k \partial z_{p_1} \cdots \partial z_{p_{m-1}}} a_{p_1} \cdots a_{p_{m-1}}   \bigg)_{1 \leq j,k \leq n},
\end{align*}
   where $f(z) = (f_1(z), f_2(z), \cdots f_n(z))'$ and $ a= (a_1, a_2, \cdots a_n)'\in \mathbb{C}^n.$

     A mapping $f: \mathbb{B} \rightarrow X$ is said to be biholomorphic if its inverse exists and holomorphic on $f(\mathbb{B})$. A mapping $f\in \mathcal{H}(\mathbb{B})$ is called locally biholomorphic if, for every $z\in \mathbb{B}$, the Fr\'{e}chet derivative $Df(z)$ is invertible and its inverse is bounded. Analogous to the one-dimensional case, normalization for $f$ is defined by the conditions $f(0)=0$ and $Df(0)=I$,  where $I$ denotes the identity operator on $X$. For each $z\in X\setminus\{0\}$, consider the set
    $$ T_z = \left\{ l_z \in L(X,\mathbb{C}) : l_z(z) = \| z\|, \| l_z \| = 1 \right\},$$
    where $L(X, Y)$ denotes the set of continuous linear operators from $X$ into a complex Banach space $Y$.  By the Hahn-Banach theorem, the set $T_z$ is non-empty.

   Liu and Liu \cite{Liu} introduced the following class:
\begin{defn}\cite{Liu}\label{defn1C}
    Suppose $\alpha \in [0,1)$ and $f: \mathbb{B} \rightarrow X$ be a normalized locally biholomorphic mapping. If
    $$ \RE \left\{l_z [(D f(z))^{-1} ( D^2 f(z) (z^2) + D f(z) (z) )]  \right\} \geq \alpha \| z \|, \quad  l_z \in T_z, \;z \in \mathbb{B}\setminus \{0 \},$$
    then $f$ is called a quasi convex mapping of type $B$ and order $\alpha$ on $\mathbb{B}$.

    If $\mathbb{B} = \mathbb{U}^n$ and $X = \mathbb{C}^n$, then the above condition reduces to
     $$ \left\vert \frac{q_k(z)}{z_k} - \frac{1}{2 \alpha} \right\vert < \frac{1}{2\alpha}, \quad \forall z \in \mathbb{U}^n \setminus \{0 \},$$
   where $$q(z) = ( q_1(z), q_2(z), \cdots , q_n(z))' = (D f(z))^{-1} ( D^2 f(z) (z^2) + D f(z) (z) )$$
   is a column vector in $\mathbb{C}^n$ and $k$ satisfies $$\vert z_k \vert = \| z \| = \max_{1 \leq j \leq n} \{ \vert z_j \vert \}.$$
   In the case $\mathbb{B} = \mathbb{U}$ and $X = \mathbb{C}$,  the condition is equivalent to
   $$ \RE \bigg( 1 + \frac{ z f''(z)}{f'(z)} \bigg) > \alpha, \quad z \in \mathbb{U}. $$
\end{defn}
    We denote by  $\mathcal{C}_\alpha(\mathbb{B})$ the class of quasi convex mappings of type $B$ and order $\alpha.$ When $\alpha =0$, Definition \ref{defn1C} coincides with the definition of quasi convex mapping of type $B$, denoted by $\mathcal{C}(\mathbb{B})$,  introduced by Roper and Suffridge \cite{Rope}.

\begin{definition}
      For  a biholomorphic function $\Psi: \mathbb{U} \rightarrow \mathbb{C}$ satisfying $\RE \Psi(z) >0$, $\Psi(0)=1$, $\Psi'(0)>0$ and $\Psi''(0) \in \mathbb{R}$. Let $\mathcal{M}_\Psi$, Graham et al.~\cite{GraHamKoh} introduced the class
     $$ \mathcal{M}_\Psi = \bigg\{ p \in \mathcal{H}(\mathbb{B}) : p(0)= Dp(0)=I, \; \frac{l_z (p(z))}{\|z\|} \in \Psi(\mathbb{U}),\; z\in \mathbb{B}\setminus\{0\}, l_z \in T_z \bigg\}.  $$
     In case of $\mathbb{B} = \mathbb{U}^n$ and $X = \mathbb{C}^n$, we have
     $$ \mathcal{M}_\Psi = \bigg\{ p \in \mathcal{H}(\mathbb{B}) : p(0)= Dp(0)=I, \; \frac{ p_j (z)}{z_j} \in \Psi(\mathbb{U}), \; z\in \mathbb{U}^n\setminus\{0\} \bigg\},  $$
     where $p(z) = (p_1(z), p_2(z), \cdots, p_n(z))'$ is a column vector in $\mathbb{C}^n$ and $j$ satisfies $$\vert z_j \vert = \| z\| = \max_{1 \leq k \leq n}\{ \vert z_k \vert \}.$$
     For $\mathbb{B} = \mathbb{U}$ and $X = \mathbb{C}$,  the relation is equivalent to
    $$\mathcal{M}_\Psi = \left\{ p \in \mathcal{H}(\mathbb{U}) : p(0) =0, p'(0)=1, \frac{p(z)}{z } \in \Psi(\mathbb{U}), z \in \mathbb{U} \right\}.$$
\end{definition}

    In order to establish the main results, we impose additional geometric constraints on the function $\Psi$.
\begin{assumption}
   In addition to the existing properties, we assume that the biholomorphic function $\Psi$ satisfies $\Psi'(0)>0$, $\Psi(\mathbb{U})$ is starlike with respect to $1$ and symmetric about the real axis in the right half plane.
\end{assumption}

    Extending the Fekete-Szeg\"{o} inequality from one dimension to higher dimensions, Xu et al.~\cite{XuC} established sharp estimate for the Fekete-Szeg\"{o} functional for the classes $\mathcal{C}_\alpha(\mathbb{B})$ and  $\mathcal{C}(\mathbb{B})$, defined on the unit ball in $X$ and on the unit polydisk in $\mathbb{C}^n$. This work was further generalized for the mappings having a zero of order $k$ at $z=0$, where $k\in \mathbb{N}$~\cite{XuJia}.

    Considering the Toeplitz determinants, Giri and Kumar~\cite{GirKum2} derived sharp bounds of $\vert T_{2,2}(f)\vert$ and $\vert T_{3,1}(f)\vert$ for the class $\mathcal{C}_\alpha(\mathbb{B})$; however, the case $\vert T_{2,3}(f)\vert$ remained unaddressed. Focusing on this problem, in section~\ref{Sec3}, we obtain sharp estimates of $\vert T_{2,3}(f)\vert$ for a class of holomorphic mappings on the unit ball in $X$ and on the unit polydisk in $\mathbb{C}^n$. As a consequence, these results yield bounds for the classes $\mathcal{C}_\alpha(\mathbb{B})$ and  $\mathcal{C}(\mathbb{B})$, thereby  extending the study of Toeplitz determinants for convex mappings from one dimension to higher dimensions.
\section{Preliminary Lemmas}

    We require the following lemmas to establish the main results.

\begin{lemma}\cite{Efr}\label{lemma1}
     If $\omega(z)= \sum_{n=1}^\infty c_n z^n \in \mathcal{B}_0$, then for any $\lambda \in \mathbb{C}$,
     $$ \vert c_2 + \lambda c_1^2 \vert \leq \max \{1, \vert \lambda\vert \}. $$
\end{lemma}

\begin{lemma}\cite{ProSzy}\label{lemma2}
    If $\omega(z) = \sum_{n=1}^\infty c_n z^n \in \mathcal{B}_0$ and $(\nu_1 , \nu_2) \in \cup_{i=1}^7 \Theta_i$, then
    $$ \vert c_3 + \nu_1 c_1 c_2 +\nu_2 c_1^3 \vert \leq \vert \nu_2 \vert ,$$
    where
\begin{align*}
  \Theta_1 &= \left\{ (\nu_1 , \nu_2) : \vert \nu_1 \vert \leq \frac{1}{2}, \vert \nu_2\vert \leq 1 \right\},  \\
   \Theta_2 &= \left\{ (\nu_1 , \nu_2) : \frac{1}{2} \leq \vert \nu_1\vert \leq 2,\; \frac{4}{27} (\vert \nu_1\vert + 1)^3 - (\vert \nu_1\vert + 1) \leq \nu_2 \leq 1 \right\}, \\
   \Theta_3 &= \left\{ (\nu_1 , \nu_2) : \vert \nu_1\vert \leq \frac{1}{2}, \nu_2 \leq -1 \right\}, \;\; \Theta_4 = \left\{ (\nu_1 , \nu_2) : \vert \nu_1\vert \geq \frac{1}{2}, \nu_2 \leq -\frac{2}{3} (\vert \nu_1\vert + 1) \right\}, \\
     \Theta_5 &= \bigg\{ (\nu_1, \nu_2) : \vert \nu_1 \vert \leq 2, \; \nu_2 \geq 1 \bigg\}, \;\;\Theta_6 = \bigg\{ (\nu_1, \nu_2) : 2 \leq \vert \nu_1 \vert \leq 4, \; \nu_2 \geq \frac{1}{12} (\nu_1^2 + 8) \bigg\},
\end{align*}
  and
   $$   \Theta_7 = \bigg\{ (\nu_1, \nu_2) : \vert \nu_1 \vert \geq 4, \; \nu_2 \geq \frac{2}{3} (\vert \nu_1\vert - 1) \bigg\}. $$
\end{lemma}

\begin{lemma}\label{lm3}
    Let $g\in \mathcal{H}(\mathbb{U})$ be a biholomorphic function such that $g(0)=\Psi(0)$ and $g(\mathbb{U})\subset \Psi(\mathbb{U})$. If $( r_1, r_2) \in \bigcup_{i=1}^7\Theta_i$, then
    $$ \Big\vert 6 (g'(0))^3+9 g'(0) g''(0)+ 2 g'''(0) \Big \vert \leq \Big\vert 6 (\Psi'(0))^3 + 9 \Psi'(0) \Psi''(0) + 2 \Psi'''(0) \Big\vert, $$
    where
    $$ r_1 = \frac{3 (\Psi'(0))^2 +  2 \Psi''(0)}{2 \Psi'(0)} \;\; \text{and}\;\; r_2 = \frac{6 (\Psi'(0))^3 + 9 \Psi'(0) \Psi''(0) + 2 \Psi'''(0)}{12 \Psi'(0)} .$$
\end{lemma}
\begin{proof}
   The premises on $g$ and $\Psi$ lead to $g \prec \Psi$. By the definition of subordination,  there exist a function $\omega(z)=\sum_{n=1}^\infty c_n z^n \in \mathcal{B}_0$ satisfying
   $$ g(z)=  \Psi(\omega(z)). $$
   Comparing the coefficients of like powers of $z$, using the Taylor series expansions for $g$, $\Psi$ and $\omega$, yields
\begin{align*}
     g'(0)&= \Psi'(0) c_1, \quad \frac{g''(0)}{2}=  \Psi'(0) c_2 +\frac{\Psi''(0)}{2} c_1^2
\end{align*}
   and
   $$ \frac{g'''(0)}{6}= \frac{\Psi'''(0)}{6} c_1^3 + \Psi''(0) c_1 c_2 + \Psi'(0)c_3.$$
   In view of these identities, it follows that
\begin{equation}\label{lm2eq1}
   \bigg\vert 6 (g'(0))^3+ 9 g'(0) g''(0) +2 g'''(0) \bigg \vert  = 12\Psi'(0) \Big\vert c_3 + r_1 c_1 c_2 + r_2 c_1^3  \Big\vert.
\end{equation}
   Utilizing Lemma~\ref{lemma2} in (\ref{lm2eq1}), we derive the required result.
\end{proof}
\section{For the class $\mathcal{C}(\Psi)$}\label{1st}
  The following result provide the sharp bound of $\vert T_{2,3}(f)\vert$ when $f\in \mathcal{C}(\Psi)$.
\begin{theorem}
    If $f \in \mathcal{C}(\Psi)$ such that $\vert \Psi''(0)+2 (\Psi'(0))^2 \vert \geq 2 \Psi'(0)$ and $(r_1, r_2) \in \bigcup_{i=1}^7\Theta_i$, then
    $$ \vert T_{2,3}(f)\vert \leq \frac{1}{144}\left( 2 (\Psi'(0))^2 + \Psi''(0) \right)^2+ \frac{1}{576} \bigg((\Psi'(0))^3+ \frac{3 \Psi'(0) \Psi''(0)}{2}+\frac{\Psi'''(0)}{3}  \bigg)^2 ,$$
    where
    $$ r_1 = \frac{1}{2 \Psi'(0)}\bigg( 3 (\Psi'(0))^2+ 2 \Psi''(0) \bigg), \;\; r_2= \frac{1}{2 \Psi'(0)}\bigg(  (\Psi'(0))^3+ \frac{3 \Psi'(0) \Psi''(0)}{2}+\frac{\Psi'''(0)}{3} \bigg). $$
    The estimate is sharp.
\end{theorem}\label{thmC}
\begin{proof}
     Let $f\in \mathcal{C}(\Psi)$ be of the form (\ref{one}), then there exists a Schwarz function $\omega(z) =\sum_{n=1}^\infty c_n z^n \in \mathcal{B}_0$ such that
     $$ 1+ \frac{z f''(z)}{f'(z)} = \Psi(\omega(z)). $$
     Comparison of same powers of $z$ after Taylor series expansions of $f$, $\Psi$ and $\omega$ yields
\begin{equation}\label{eq1}
    a_3 =  \frac{\Psi'(0)}{6} \bigg( c_2 + \bigg( \Psi'(0)+ \frac{\Psi''(0)}{2 \Psi'(0)} \bigg)c_1^2 \bigg) \;\;\;\;
\end{equation}
    and
\begin{equation}\label{eq2}
    a_4=\frac{\Psi'(0)}{12}  \left( c_3 + r_1 c_1 c_2 + r_2  c_1^3  \right) .
\end{equation}
      According to the premises, $\Psi$ satisfies $\vert \Psi''(0)+2 (\Psi'(0))^2 \vert \geq 2 \Psi'(0)$ and $(r_1, r_2) \in \bigcup_{i=1}^7\Theta_i$. Consequently, by applying Lemma~\ref{lemma1} and Lemma~\ref{lemma2} to (\ref{eq1}) and (\ref{eq2}), respectively, we obtain
\begin{equation}\label{eq3}
     \vert a_3 \vert \leq \frac{2 (\Psi'(0))^2 + \Psi''(0)}{12}  \;\; \text{and} \;\; \vert a_4 \vert \leq \frac{1}{24} \bigg((\Psi'(0))^3+ \frac{3 \Psi'(0) \Psi''(0)}{2}+\frac{\Psi'''(0)}{3}  \bigg).
\end{equation}
     From (\ref{Tplz}), it follows that
\begin{equation}\label{FnTplz}
     \left\vert T_{2,3}(f) \right\vert  \leq \vert a_2\vert^2 + \vert a_3\vert^2.
\end{equation}
    Using the estimates for $\vert a_2 \vert$ and $\vert a_3 \vert$ from (\ref{eq3}) in (\ref{FnTplz}), we get
\begin{equation}\label{extrem}
\left\vert T_{2,3}(f) \right\vert \leq \frac{( 2 (\Psi'(0))^2 + \Psi''(0) )^2}{144}+ \frac{1}{576} \bigg((\Psi'(0))^3+ \frac{3 \Psi'(0) \Psi''(0)}{2}+\frac{\Psi'''(0)}{3}  \bigg)^2.
\end{equation}

    To show the sharpness, we consider the function $f_\Psi : \mathbb{U}\rightarrow \mathbb{C}$, given by
\begin{equation*}
    1 + \frac{f''_\Psi(z)}{f'_\Psi(z)} =   \Psi(i z).
\end{equation*}
    Clearly $f_\Psi \in \mathcal{C}(\Psi)$ and  for this function, we have
    $$ a_3 = -\frac{1}{6}\bigg( (\Psi'(0))^2 + \frac{\Psi''(0)}{2} \bigg)\;\; \text{and}\;\; a_4 =- \frac{i}{24}\bigg( (\Psi'(0))^3 + \frac{3 \Psi'(0)\Psi''(0)}{2}+ \frac{\Psi'''(0)}{3} \bigg),$$
     which together with (\ref{Tplz}) shows that equality case holds in (\ref{extrem}) for the function $f_\Psi$.
\end{proof}
   For $\Psi(z)=(1+z)/(1-z)$, $\Psi(z)=(1+(1-2\alpha)z)/(1-z)$ and $\Psi(z)=((1+z)/(1-z))^\beta$, the class $\mathcal{C}(\Psi)$ corresponds to the classes $\mathcal{C}$, $\mathcal{C}(\alpha)$ and $\mathcal{CC}(\beta)$, respectively. Consequenlty, the following results are obtained as immediate consequences of Theorem~\ref{thmC} for these subclasses.
\begin{corollary}\label{crl1}
    If $f\in \mathcal{C}$, then $\vert T_{2,3}(f)\vert \leq 2.$
\end{corollary}
\begin{corollary}\label{crl2}
    If $f\in \mathcal{C}(\alpha)$, then
      $$  \vert T_{2,3}(f)\vert \leq \frac{(1 - \alpha )^2 (3- 2 \alpha)^2}{9} +\frac{(1- \alpha)^2 (2- \alpha)^2 (3- 2 \alpha)^2}{36}, \quad \alpha \in [0,1] .$$
\end{corollary}
\begin{corollary}
    If $f\in \mathcal{CC}(\beta)$, then
      $$  \vert T_{2,3}(f)\vert \leq  \beta ^4+\frac{\beta ^2 (1+17 \beta ^2)^2 }{324},\quad \beta \in [2/3,1] .$$
\end{corollary}
\section{In higher dimensions}\label{Sec3}
   This section establishes sharp estimates of $\vert T_{2,3}(f) \vert$ for a subclass of holomorphic mappings on the unit ball of a complex Banach space and the unit polydisc in $\mathbb{C}^n$.
\begin{theorem}\label{thmB}
Let $f\in \mathcal{H}(\mathbb{B},\mathbb{C})$ with $f(0)=1$, $f(z)\neq 0$, $z\in \mathbb{B}$ and suppose $F(z)= z f(z)$. If $(DF(z))^{-1}(D^2 F(z)(z^2)+ DF(z)(z))\in \mathcal{M}_\Psi$ such that
    $$ \lvert \Psi''(0) + 2 (\Psi'(0))^2 \rvert \geq 2  \Psi'(0) \;\; \quad \text{and} \;\; ( r_1, r_2) \in \bigcup_{i=1}^7\Theta_i,$$
    then
\begin{equation*}
\begin{aligned}
   \bigg\vert \bigg(\frac{ l_z (D^3 F(0)(z^3))}{3! \|z\|^3} &\bigg)^2  - \bigg(\frac{ l_z (D^4 F(0)(z^4))}{4! \|z\|^4} \bigg)^2\bigg\vert \leq  \frac{1}{144}\left( 2 (\Psi'(0))^2 + \Psi''(0) \right)^2 \\
   &\;\;\;\;+ \frac{1}{576} \bigg((\Psi'(0))^3+ \frac{3 \Psi'(0) \Psi''(0)}{2}+\frac{\Psi'''(0)}{3}  \bigg)^2 , \;\; z\in \mathbb{B}\setminus\{0\}  .
\end{aligned}
\end{equation*}
   The bound is sharp.
\end{theorem}
\begin{proof}
   Let $g : \mathbb{U} \rightarrow \mathbb{C}$ be defined by
\begin{align*}
    g(\zeta)=
\left\{
\begin{array}{ll}
     \dfrac{l_z ((DF(\zeta z_0))^{-1}(D^2 F(\zeta z_0)(\zeta z_0^2)+ DF(\zeta z_0)(\zeta z_0)) )}{\zeta}, & \zeta \neq 0 \\ \ \\
     1, & \zeta =0,
\end{array}
\right.
\end{align*}
  where $z_0= \frac{z}{\|z\|}$ for fix $z\in X\setminus \{0\}$. Then $g\in \mathcal{H}(\mathbb{U})$ and $g(0)=\Psi(0)=1$. Since $(DF(z))^{-1}(D^2 F(z)(z^2)+ DF(z)(z))\in \mathcal{M}_\Psi$, it follows that
\begin{align*}
    g(\zeta) &= \frac{l_z((DF(\zeta z_0))^{-1}(D^2 F(\zeta z_0)(\zeta z_0^2)+ DF(\zeta z_0)(\zeta z_0)))}{\zeta} \\
            &=  \frac{l_{z_0} ((DF(\zeta z_0))^{-1}(D^2 F(\zeta z_0)(\zeta z_0^2)+ DF(\zeta z_0)(\zeta z_0)))}{\zeta} \\
            &= \frac{l_{\zeta z_0} ((DF(\zeta z_0))^{-1}(D^2 F(\zeta z_0)(\zeta z_0^2)+ DF(\zeta z_0)(\zeta z_0)))}{\| \zeta z_0\|} \in \Psi(\mathbb{U}).
\end{align*}
   From $F(z)= z f(z)$, we deduce that
\begin{equation}\label{F=zf}
    D^2 F(z)(z^2)+D F(z)(z)= ( D^2 f(z)(z^2)+  3 D f(z) (z) + f(z))z.
\end{equation}
   Following the same procedure as in \cite[Theorem 7.1.14]{GraKoh}, we obtain
\begin{equation}\label{FFinverse}
    (D F(z))^{-1} = \frac{1}{f(z)} \bigg( I - \frac{\frac{z Df(z)}{f(z)}}{1+ \frac{D f(z)z}{f(z)}} \bigg).
\end{equation}
   Equations (\ref{F=zf}) and (\ref{FFinverse}) together yield
\begin{equation}\label{3.3}
    (DF(z))^{-1}(D^2 F(z)(z^2)+ DF(z)(z)) = \frac{D^2 f(z)(z^2)+  3 D f(z) (z) + f(z)}{f(z)+ D (z)(z)} z.
\end{equation}
     Consequently, we get
\begin{equation*}\label{Keep}
   \frac{l_z( (DF(z))^{-1}(D^2 F(z)(z^2)+ DF(z)(z)))}{\| z\|} = \frac{D^2 f(z)(z^2)+  3 D f(z) (z) + f(z)}{f(z)+ D (z)(z)}.
\end{equation*}
    Therefore, we obtain
\begin{align*}
    g(\zeta) &= \frac{l_{\zeta z_0} ((DF(\zeta z_0))^{-1}(D^2 F(\zeta z_0)(\zeta z_0^2)+ DF(\zeta z_0)(\zeta z_0)))}{\| \zeta z_0\|} \\
               &=  \frac{D^2 f(\zeta z_0)((\zeta z_0)^2)+  3 D f(\zeta z_0) (\zeta z_0) + f(\zeta z_0)}{f(\zeta z_0)+ D (\zeta z_0)(\zeta z_0)}.
\end{align*}
     Expanding $f$ and $g$ in Taylor series leads to
\begin{align*}
     &\bigg(1+ g'(0) \zeta + \frac{g''(0)}{2} \zeta^2 + \frac{g'''(0)}{6} \zeta^3 +\cdots\bigg) \bigg(1+ 2 D f(0)(z_0) \zeta + \frac{3}{2} D^2 f(0)(z_0^2) \zeta^2 \\
     &+ \frac{2}{3}D^3 f(0)(z_0^3) \zeta^3+\cdots \bigg)= \bigg( 1+ 4 D f(0)(z_0)\zeta+\frac{9}{2}D^2 f(0)(z_0^2)\zeta^2 + \frac{8}{3}D^3 f(0)(z_0^3)\zeta^3+ \cdots\bigg).
\end{align*}
    Equating the homogeneous terms in the series expansions of both sides yields
\begin{align*}
    D f(0)(z_0)&=  \frac{g'(0)}{2}, \quad \frac{D^2 f(0)(z_0^2)}{2} =\frac{1}{6}\bigg(\frac{g''(0)}{2}+ (g'(0))^2\bigg)
\end{align*}
    \text{and}
\begin{align*}
    \frac{D^3 f(0)(z_0^3)}{6} &=\frac{1}{72} \left(3 (g'(0))^3+\frac{9 g'(0) g''(0)}{2}+g'''(0)\right).
\end{align*}
   That is
\begin{align*}
    D f(0)(z)&=  \frac{g'(0)}{2}\|z\|, \quad \frac{D^2 f(0)(z^2)}{2} =\frac{1}{6}\bigg(\frac{g''(0)}{2}+ (g'(0))^2\bigg)\|z\|^2
\end{align*}
and
\begin{equation}\label{thiseq}
    \frac{D^3 f(0)(z^3)}{6} =\frac{1}{72} \left(3 (g'(0))^3+\frac{9 g'(0) g''(0)}{2}+g'''(0)\right)\|z\|^3.
\end{equation}
   Again, by the fact $F(z)=z f(z)$, we have
   $$ \frac{D^4 F(0)(z^4)}{4!} = \frac{D^3 f(0)(z^3)}{3!}z,$$
   which immediately gives
   $$ \frac{l_z(D^4 F(0)(z^4))}{4!} = \frac{D^3 f(0)(z^3)}{3!}\|z\|. $$
  Using~(\ref{thiseq}) in the above identity, we get
         $$ \frac{l_z(D^4 F(0)(z^4))}{4!} =\frac{1}{144} \left(6 (g'(0))^3+9 g'(0) g''(0)+2 g'''(0)\right)\|z\|^4.$$
  As $g\prec \Psi$,  Lemma~\ref{lm3} implies
\begin{equation}\label{a4BB}
      \bigg \vert \frac{l_z(D^4 F(0)(z^4))}{4! \|z\|^4} \bigg\vert \leq \frac{1}{24} \left\vert (\Psi'(0))^3+\frac{3 \Psi'(0) \Psi''(0)}{2}+\frac{ \Psi'''(0)}{3}\right\vert.
\end{equation}
  Furthermore, Xu et al. \cite[Theorem 3.1]{XuC} established the result for $\lambda \in \mathbb{C}$ that
\begin{equation*}\label{FSBC}
\begin{aligned}
   \bigg\vert  \frac{ l_z (D^3 F(0) (z^3))}{3! \vert\vert z \vert\vert^3} &-  \lambda \bigg(\frac{ l_z (D^2 F(0) (z^2))}{2! \vert\vert z \vert\vert^2} \bigg)^2 \bigg\vert \\
   & \leq \frac{\vert \Psi'(0) \vert}{6} \max \left\{ 1,  \left\lvert \frac{1}{2} \frac{\Psi''(0)}{\Psi'(0)} + \bigg(1 - \dfrac{3}{2} \lambda \bigg)  \Psi'(0) \right\rvert \right\} , \;\;  z \in \mathbb{B}\setminus \{0 \}.
\end{aligned}
\end{equation*}
    Given that $ \lvert \Psi''(0) + 2 (\Psi'(0))^2 \rvert \geq 2  \Psi'(0)$, it follows from the above inequality that
\begin{equation}\label{a3BC}
     \bigg\vert  \dfrac{ l_z (D^3 F(0) (z^3))}{3! \vert\vert z \vert\vert^3} \bigg\vert \leq \dfrac{ \Psi'(0)}{6}   \left( \dfrac{1}{2} \dfrac{\Psi''(0)}{\Psi'(0)} +  \Psi'(0) \right).
\end{equation}
    Incorporating the bounds from (\ref{a4BB}) and (\ref{a3BC}) into the inequality
\begin{align*}
     \bigg\vert \bigg( \frac{ l_z (D^3 F(0) (z^3))}{3! \vert\vert z \vert\vert^3} \bigg)^2 &-    \bigg(   \frac{ l_z ( D^4 F(0) (z^4))  }{4! \| z\|^4}  \bigg)^2 \bigg\vert \leq  \bigg\vert  \frac{ l_z (D^3 F(0) (z^3))}{3! \vert\vert z \vert\vert^3} \bigg\vert^2  +   \bigg\vert \frac{ l_z ( D^4 F(0) (z^4)) }{4! \| z\|^4} \bigg\vert^2
\end{align*}
   leads to the desired result.

   Sharpness of the obtained bound can be verified by the mapping
   $$ D F_\Psi(z) = I \exp \int_{0}^{l_u(z)} \bigg(\frac{\Psi(i t)-1}{t}\bigg)dt, \quad z\in \mathbb{B}, \quad \|u\|=1.$$
   It follows that $(DF_\Psi(z))^{-1}(D^2 F_\Psi(z)(z^2)+ DF_\Psi(z)(z))\in \mathcal{M}_\Psi$, and a straightforward computation gives
   $$ \frac{D^3 F(0)(z^3)}{3!} = -\frac{1}{6}\bigg( (\Psi'(0))^2 + \frac{\Psi''(0)}{2} \bigg) (l_u(z))^2 z$$
   and
   $$ \frac{D^4 F(0)(z^4)}{4!} =- \frac{i}{24}\bigg( (\Psi'(0))^3 + \frac{3 \Psi'(0)\Psi''(0)}{2}+ \frac{\Psi'''(0)}{3} \bigg) (l_u(z))^3 z,$$
   which immediately provides
   $$ \frac{l_z (D^3 F(0)(z^3))}{3!} = -\frac{1}{6}\bigg( (\Psi'(0))^2 + \frac{\Psi''(0)}{2} \bigg) (l_u(z))^2 \|z\|$$
   and
   $$ \frac{l_z(D^4 F(0)(z^4))}{4!} =- \frac{i}{24}\bigg( (\Psi'(0))^3 + \frac{3 \Psi'(0)\Psi''(0)}{2}+ \frac{\Psi'''(0)}{3} \bigg) (l_u(z))^3 \|z\|.$$
   On setting $z=r u$ $(0< r< 1)$, we obtain
   $$ \frac{l_z (D^3 F(0)(z^3))}{3! \|z\|^3} = -\frac{1}{6}\bigg( (\Psi'(0))^2 + \frac{\Psi''(0)}{2} \bigg)$$
   and
   $$ \frac{l_z(D^4 F(0)(z^4))}{4! \|z\|^4} =- \frac{i}{24}\bigg( (\Psi'(0))^3 + \frac{3 \Psi'(0)\Psi''(0)}{2}+ \frac{\Psi'''(0)}{3} \bigg) .$$
  Thus for the mapping $F_\Psi$, we have
\begin{align*}
  \bigg\vert\frac{l_z (D^3 F(0)(z^3))}{3! \|z\|^3} -\frac{l_z(D^4 F(0)(z^4))}{4! \|z\|^4}\bigg\vert &= \frac{1}{144}\left( 2 (\Psi'(0))^2 + \Psi''(0) \right)^2 \\
  &+ \frac{1}{576} \bigg((\Psi'(0))^3+ \frac{3 \Psi'(0) \Psi''(0)}{2}+\frac{\Psi'''(0)}{3}  \bigg)^2,
\end{align*}
  which completes the proof.
\end{proof}

\begin{theorem}\label{ThmUn1}
   Let $f\in \mathcal{H}(\mathbb{U}^n,\mathbb{C})$ with $f(0)=1$, $f(z)\neq 0$, $z\in \mathbb{U}^n$ and suppose $F(z)= z f(z)$. If $(DF(z))^{-1}(D^2 F(z)(z^2)+ DF(z)(z))\in \mathcal{M}_\Psi$ such that
    $$  \lvert \Psi''(0) + 2 (\Psi'(0))^2 \rvert \geq 2 \Psi'(0) \;\;\text{and}\;\; (r_1,r_2)\in \bigcup_{i=1}^7 \Theta_i ,$$
   then
\begin{equation}\label{mnresult2}
\begin{aligned}
\left.
\begin{array}{ll}
     &  \bigg\|  \dfrac{1}{4!}D^4 F(0) \bigg(z^3,  \dfrac{D^4 F(0)(z^4)}{4!} \bigg)  - \dfrac{1}{3 !} D^3 F(0) \bigg( z^2, \dfrac{D^3 F(0)(z^3)}{3!} \bigg)  \bigg\|  \\
   &\quad\quad\quad\quad\quad\quad\quad\quad\quad\quad\quad\leq  \dfrac{\|z\|^7}{576} \bigg\vert(\Psi'(0))^3 + \dfrac{3 \Psi'(0)\Psi''(0)}{2}+ \dfrac{\Psi'''(0)}{3} \bigg\vert^2   \\
 & \quad\quad\quad\quad\quad\quad\quad\quad\quad\quad\quad+   \dfrac{ (\Psi'(0))^2 \|z\|^5}{36}   \left( \dfrac{1}{2} \dfrac{\Psi''(0)}{\Psi'(0)} +   \Psi'(0) \right)^2, \; z \in \mathbb{U}^n  .
\end{array}
\right\}
\end{aligned}
\end{equation}
   The estimate is sharp.
\end{theorem}
\begin{proof}
  For $z\in \mathbb{U}^n \setminus \{0\}$, let $z_0 = \frac{z}{\| z\|}$. Define the function $g_k : \mathbb{U} \rightarrow \mathbb{C}$ such that
\begin{equation}\label{hkzeta}
   g_k (\zeta) =
\left\{
\begin{array}{ll}
     \dfrac{h_k (\zeta z_0) \| z_0 \|}{\zeta z_k}, & \zeta \neq 0,\\
     1 , & \zeta =0,
\end{array}
\right.
\end{equation}
    where $h(z) = (D F(z))^{-1} ( D^2 F(z) (z^2) + D F(z) (z))$ and $k$ satisfies $\vert z_k \vert = \| z \| = \max_{1\leq j \leq n} \{ z_j \}$. Since $(D F(z))^{-1} ( D^2 F(z) (z^2) + D F(z) (z)) \in \mathcal{M}_\Psi$, it follows that $g_k (\zeta) \in h (\mathbb{U})$.
     A direct consequence of  (\ref{3.3}) yields that
    $$ g_k(\zeta) =   \frac{ D^2 f(\zeta z_0) ((\zeta z_0)^2 ) + 3 D f(\zeta z_0)(\zeta z_0) + f(\zeta z_0) }{ f(\zeta z_0)  + D f(\zeta z_0) ( \zeta z_0)} .  $$
   Expanding of $f$ and $g_k$ in Taylor series about $\zeta$, we obtain
\begin{align*}
     &\bigg(1+ g'_k(0) \zeta + \frac{g''_k(0)}{2} \zeta^2 + \frac{g'''_k(0)}{6} \zeta^3 +\cdots\bigg) \bigg(1+ 2 D f(0)(z_0) \zeta + \frac{3}{2} D^2 f(0)(z_0^2) \zeta^2 \\
     &+ \frac{2}{3}D^3 f(0)(z_0^3) \zeta^3+\cdots \bigg)= \bigg( 1+ 4 D f(0)(z_0)\zeta+\frac{9}{2}D^2 f(0)(z_0^2)\zeta^2 + \frac{8}{3}D^3 f(0)(z_0^3)\zeta^3+ \cdots\bigg).
\end{align*}
    A comparison of the homogeneous expansions on both sides of the equality leads to
\begin{align*}
    D f(0)(z_0)&=  \frac{g'_k(0)}{2}, \quad \frac{D^2 f(0)(z_0^2)}{2} =\frac{1}{6}\bigg(\frac{g''_k(0)}{2}+ (g'_k(0))^2\bigg)
\end{align*}
    \text{and}
\begin{equation}\label{a3Un}
    \frac{D^3 f(0)(z_0^3)}{6} =\frac{6 (g'_k(0))^3+ 9 g'_k(0) g''_k(0)+ 2 g'''_k(0)}{144}.
\end{equation}
     In addition, from the relation $F(z_0) =  z_0 f(z_0)$, we have
\begin{equation}\label{use}
    \frac{D^3 F_k(0) (z_0^3)}{3!} = \frac{D^2 f(0) (z_0^2)}{2!} \frac{z_k}{ \| z\|}\;\;\; \text{and} \;\; \;  \frac{D^4 F_k(0) (z_0^4)}{4!} =\frac{ D^3 f(0) (z_0^3)}{3! }  \frac{z_k}{ \| z\|}.
\end{equation}
    Using (\ref{use}), we deduce that
\begin{equation}\label{Second2}
\begin{aligned}
\left.
\begin{array}{ll}
    \bigg\vert  \dfrac{1}{4 !} D^4 F_k(0) \bigg( z_0^3, \dfrac{D^4 F(0)(z_0^4)}{4!}& \bigg) \dfrac{\| z\| }{z_k} \bigg\vert
    \\&=\Big\vert  \dfrac{1}{4 !}D^4 F_k(0) \bigg( z_0^3, \dfrac{ D^3 f(0) (z_0^3)}{3! }   z_0 \bigg) \dfrac{\| z\| }{z_k} \bigg\vert \\
       &= \Big\vert \dfrac{ D^3 f(0) (z_0^3)}{3! } \dfrac{1}{4 !} D^4 F_k(0) ( z_0^3,  z_0 ) \dfrac{\| z\| }{z_k} \bigg\vert  \\
       &= \Big\vert  \dfrac{ D^3 f(0) (z_0^3)}{3! } \dfrac{1}{4 !} D^4 F_k(0) ( z_0^4 ) \dfrac{\| z\| }{z_k} \bigg\vert\\
       & = \Big\vert \bigg(  \dfrac{ D^3 f(0) (z_0^3)}{3! }  \bigg)^2 \bigg\vert.
\end{array}
\right\}
\end{aligned}
\end{equation}
   By applying (\ref{a3Un}) to (\ref{Second2}), we get
\begin{equation*}
    \bigg\vert  \frac{1}{4 !} D^4 F_k(0) \bigg( z_0^3, \frac{D^4 F(0)(z_0^4)}{4!} \bigg) \frac{\| z\| }{z_k} \bigg\vert  = \bigg\vert\frac{6 (g'_k(0))^3+9 g'_k(0) g''_k(0)+2 g'''_k(0)}{144}\bigg\vert^2.
\end{equation*}
   In view of the hypothesis $ (r_1,r_2)\in \bigcup_{i=1}^7 \Theta_i ,$ Lemma~\ref{lm3} gives
\begin{equation*}
    \bigg\vert  \frac{1}{4 !} D^4 F_k(0) \bigg( z_0^3, \frac{D^4 F(0)(z_0^4)}{4!} \bigg) \frac{\| z\| }{z_k} \bigg\vert  \leq \frac{1}{576} \bigg\vert(\Psi'(0))^3 + \frac{3 \Psi'(0)\Psi''(0)}{2}+ \frac{\Psi'''(0)}{3} \bigg\vert^2.
\end{equation*}
   If $z_0 \in \partial_0 \mathbb{U}^n$, then it follows that
\begin{equation*}
    \bigg\vert  \frac{1}{4 !} D^4 F_k(0) \bigg( z_0^3, \frac{D^4 F(0)(z_0^4)}{4!} \bigg)  \bigg\vert  \leq \frac{1}{576} \bigg\vert(\Psi'(0))^3 + \frac{3 \Psi'(0)\Psi''(0)}{2}+ \frac{\Psi'''(0)}{3} \bigg\vert^2.
\end{equation*}
 Since
   $$  \frac{1}{4 !} D^4 F_k(0) \bigg( z_0^3, \frac{D^4 F(0)(z_0^4)}{4!} \bigg) , \quad k =1,2,\cdots , n,$$
   are holomorphic on $\overline{\mathbb{U}}^n$, the maximum modulus principle yields
\begin{equation}\label{a4Un}
    \bigg\vert  \frac{1}{4 !} D^4 F_k(0) \bigg( z^3, \frac{D^4 F(0)(z^4)}{4!} \bigg)  \bigg\vert  \leq \frac{\|z\|^7}{576} \bigg\vert(\Psi'(0))^3 + \frac{3 \Psi'(0)\Psi''(0)}{2}+ \frac{\Psi'''(0)}{3} \bigg\vert^2.
\end{equation}
  For $\lambda \in \mathbb{C}$, Xu et al. \cite[Theorem 3.2]{XuC} established that
\begin{equation}\label{FSUnC}
\begin{aligned}
\left.
\begin{array}{ll}
    \bigg\vert   \dfrac{ D^3 F_k(0) (z_0^3)}{3! } \dfrac{\| z\| }{z_k} -  &\lambda \dfrac{1}{2} D^2  F_k (0)  \bigg( z_0, \dfrac{D^2 F(0) (z_0^2)}{2!}\dfrac{\| z\| }{z_k} \bigg) \bigg\vert \\
   &\;\;\;\; \leq \dfrac{\vert \Psi'(0) \vert }{6} \max \bigg\{ 1  ,\left\lvert \dfrac{1}{2} \dfrac{\Psi''(0)}{\Psi'(0)} + \bigg(1 - \dfrac{3}{2} \lambda \bigg)  \Psi'(0)  \right\rvert \bigg\} .
\end{array}
\right\}
\end{aligned}
\end{equation}
    Since  $\lvert \Psi''(0) + 2 (\Psi'(0))^2 \rvert \geq 2 \Psi'(0)$, from (\ref{use}) and (\ref{FSUnC}), we get
\begin{equation}\label{a3UnC}
    \bigg\vert  \frac{ D^3 F_k(0) (z_0^3)}{3! } \dfrac{\| z\| }{z_k}   \bigg\vert =  \bigg\vert \frac{D^2 f(0) (z_0^2)}{2!} \bigg\vert   \leq \frac{ \Psi'(0)}{6}   \left( \frac{1}{2} \frac{\Psi''(0)}{\Psi'(0)} +   \Psi'(0) \right).
\end{equation}
   Following the same technique as in (\ref{Second2}), we derive
\begin{equation}\label{spl4}
     \bigg\vert  \frac{1}{3!}D^3 F_k(0) \bigg(z_0^2,  \frac{D^3 F(0)(z_0^3)}{3!} \bigg) \frac{\| z\| }{z_k} \bigg\vert =\bigg\vert   \bigg( \frac{D^2 f(0)(z_0^2)}{2!} \bigg) \bigg\vert^2.
\end{equation}
   Consequently. equations~(\ref{a3UnC}) and (\ref{spl4}) yields
\begin{equation*}
     \bigg\vert  \frac{1}{3!}D^3 F_k(0) \bigg(z_0^2,  \frac{D^3 F(0)(z_0^3)}{3!} \bigg) \frac{\| z\| }{z_k} \bigg\vert \leq  \frac{ (\Psi'(0))^2}{36}   \left( \frac{1}{2} \frac{\Psi''(0)}{\Psi'(0)} +   \Psi'(0) \right)^2.
\end{equation*}
      For $z_0 \in \partial_0 \mathbb{U}^n$, we get
\begin{equation*}
     \bigg\vert  \frac{1}{3!}D^3 F_k(0) \bigg(z_0^2,  \frac{D^3 F(0)(z_0^3)}{3!} \bigg)  \bigg\vert \leq  \frac{ (\Psi'(0))^2}{36}   \left( \frac{1}{2} \frac{\Psi''(0)}{\Psi'(0)} +   \Psi'(0) \right)^2.
\end{equation*}
  Due to the holomorphicity of the functions
   $$  \frac{1}{3!}D^3 F_k(0) \bigg(z_0^2,  \frac{D^3 F(0)(z_0^3)}{3!} \bigg)  , \quad k =1,2,\cdots , n,$$
     on the closed  unit polydisc  $\overline{\mathbb{U}}^n$,  the maximum modulus principle for holomorphic
functions implies that
\begin{equation}\label{A33Un}
     \bigg\vert  \frac{1}{3!}D^3 F_k(0) \bigg(z^2,  \frac{D^3 F(0)(z^3)}{3!} \bigg)  \bigg\vert \leq  \frac{ (\Psi'(0))^2 \|z\|^5}{36}   \left( \frac{1}{2} \frac{\Psi''(0)}{\Psi'(0)} +   \Psi'(0) \right)^2.
\end{equation}
   Using the bounds from (\ref{a4Un}) and (\ref{A33Un}) together with triangular inequality, we derive
\begin{align*}
    \bigg\vert &\frac{1}{4!}D^4 F_k(0) \bigg(z^3,  \frac{D^4 F(0)(z^4)}{4!} \bigg) - \frac{1}{3!} D^3 F_k(0) \bigg( z^2, \frac{D^3 F(0)(z^3)}{3!} \bigg) \bigg\vert \\
      &\leq \frac{\|z\|^7}{(144)^2} \bigg\vert \left(6 (\Psi'_k(0))^3+9 \Psi'_k(0) \Psi''_k(0)+2 \Psi'''_k(0)\right)\bigg\vert^2  +   \frac{ (\Psi'(0))^2 \|z\|^5}{36}   \left( \frac{1}{2} \frac{\Psi''(0)}{\Psi'(0)} +   \Psi'(0) \right)^2
\end{align*}
   for $ k=1,2, \cdots n.$ Therefore,
\begin{align*}
    \bigg\| &\frac{1}{4!}D^4 F_k(0) \bigg(z^3,  \frac{D^4 F(0)(z^4)}{4!} \bigg) - \frac{1}{3!} D^3 F_k(0) \bigg( z^2, \frac{D^3 F(0)(z^3)}{3!} \bigg) \bigg\|  \\
      &\leq \frac{\|z\|^7}{576} \bigg\vert(\Psi'(0))^3 + \frac{3 \Psi'(0)\Psi''(0)}{2}+ \frac{\Psi'''(0)}{3} \bigg\vert^2+   \frac{ (\Psi'(0))^2 \|z\|^5}{36}   \left( \frac{1}{2} \frac{\Psi''(0)}{\Psi'(0)} +   \Psi'(0) \right)^2, \;z \in \mathbb{U}^n.
\end{align*}

    To show that the estimate is sharp, consider the mapping $F_\Psi$ given by
\begin{equation}\label{extUnC}
     D F_\Psi(z) = I \exp \int_0^{z_1} \frac{\Psi(i t)-1}{t} dt.
\end{equation}
    It can be verified that $(D F_\Psi(z))^{-1} ( D^2 F_\Psi(z) (z^2) + D F_\Psi(z) (z) ) \in \mathcal{M}_\Psi$ and for $z =(r,0,\cdots,0)'$ equality is attained in (\ref{mnresult2}), thereby confirming the sharpness.
\end{proof}

\subsection{For Quasi-Convex mappings}
   If $f\in \mathcal{H}(\mathbb{B})$ and $(DF(z))^{-1}(D^2 F(z)(z^2)+ DF(z)(z))\in \mathcal{M}_\Psi$, then suitable choices of $\Psi$ lead to various subclasses of holomorphic mappings. For example, selecting $\Psi(z)=(1+z)/(1-z)$ and $\Psi(z)=(1+(1-2 \alpha)z)/(1-z)$ corresponds to the classes $\mathcal{C}_\alpha(\mathbb{B})$ and $\mathcal{C}(\mathbb{B})$, respectively. Consequently, Theorem~\ref{thmB} and Theorem~\ref{ThmUn1} yield the following bounds for the corresponding classes, which extend Corollary~\ref{crl1} and Corollary~\ref{crl2} from the one variable case to several complex variables.

\begin{corollary}
   Let $f\in \mathcal{H}[\mathbb{B},\mathbb{C}]$ and $F(z)=z f(z) \in \mathcal{C}(\mathbb{B})$. Then
\begin{equation*}
\begin{aligned}
   \bigg\vert \bigg(\frac{ l_z (D^3 F(0)(z^3))}{3! \|z\|^3} \bigg)^2 & - \bigg(\frac{ l_z (D^4 F(0)(z^4))}{4! \|z\|^4} \bigg)^2\bigg\vert \leq 2 , \;\; z\in \mathbb{B}\setminus\{0\}  .
\end{aligned}
\end{equation*}
   If $\mathbb{B}=\mathbb{U}^n$ and $X= \mathbb{C}^n$, then
  $$  \bigg\|  \frac{1}{4!}D^4 F(0) \bigg(z^3,  \frac{D^4 F(0)(z^4)}{4!} \bigg)  - \frac{1}{3 !} D^3 F(0) \bigg( z^2, \frac{D^3 F(0)(z^3)}{3!} \bigg)  \bigg\| \leq \|z\|^5 + \|z\|^7 , \; z \in \mathbb{U}^n  . $$
  All these estimates are sharp.
\end{corollary}

\begin{corollary}
   Let $f\in \mathcal{H}[\mathbb{B},\mathbb{C}]$ and $F(z)=z f(z) \in \mathcal{C}_\alpha(\mathbb{B})$. Then
\begin{equation*}
\begin{aligned}
   \bigg\vert \bigg(\frac{ l_z (D^3 F(0)(z^3))}{3! \|z\|^3} \bigg)^2 & - \bigg(\frac{ l_z (D^4 F(0)(z^4))}{4! \|z\|^4} \bigg)^2\bigg\vert \leq \frac{(1 - \alpha )^2 (3- 2 \alpha)^2}{9}\\
   & \quad\quad\quad\quad\quad+\frac{(1- \alpha)^2 (2- \alpha)^2 (3- 2 \alpha)^2}{36} , \;\; z\in \mathbb{B}\setminus\{0\}  .
\end{aligned}
\end{equation*}
   If $\mathbb{B}=\mathbb{U}^n$ and $X= \mathbb{C}^n$, then
\begin{equation*}
\begin{aligned}
\begin{array}{ll}
     &  \bigg\|  \dfrac{1}{4!}D^4 F(0) \bigg(z^3,  \dfrac{D^4 F(0)(z^4)}{4!} \bigg)  - \dfrac{1}{3 !} D^3 F(0) \bigg( z^2, \dfrac{D^3 F(0)(z^3)}{3!} \bigg)  \bigg\|  \\
   &\quad\quad\quad\quad\leq \dfrac{\|z\|^5(1 -\alpha )^2 (3- 2 \alpha)^2 }{9}+ \dfrac{\|z\|^7(1 -\alpha )^2 (2 -\alpha)^2  ( 3- 2 \alpha)^2 }{36}  , \; z \in \mathbb{U}^n  .
\end{array}
\end{aligned}
\end{equation*}
  All these estimates are sharp.
\end{corollary}

\section*{Declarations}

\subsection*{Conflict of interest}
	The author declare that he has no conflict of interest.
\subsection*{Data Availability} Not Applicable.

\end{document}